\documentclass[a4paper,twoside]{article}
\usepackage{amssymb}
\usepackage{amsmath}
\usepackage{fancyhdr}
\usepackage{amsthm}
\usepackage{rotate}
\usepackage{graphicx}
\usepackage[T1]{fontenc}
\usepackage{afterpage}  %

\newtheorem{theorem}{Theorem}
\newtheorem{lemma}{Lemma}

\theoremstyle{definition}
\newtheorem{example}{Example}

\begin{document}
\afterpage{\rhead[]{\thepage} \chead[\small  I. I. Deriyenko 
]{\small Rectangles in latin squares} \lhead[\thepage]{} }                  

\begin{center}
\vspace*{2pt} {\Large \textbf{Rectangles in latin squares}}\\[26pt]
{\large \textsf{\emph{Ivan I. Deriyenko }}} 
\\[26pt]
\end{center}
{\footnotesize{\bf Abstract.}  To get another from a given latin square, we have to change at least 4 entries.
We show how to find these entries and how to change them.}
\date{} 


\noindent{\bf 1.}
The {\em distance} between two latin squares $||a_{ij}||$ and $||b_{ij}||$ of the same size $n>2$ is equal to the number of cells in which the corresponding elements $a_{ij}$ and $b_{ij}$ are not equal. The minimal distance between two latin squares is equal $4$. Moreover, for any $n>3$ there exists two latin squares of size $n$ which differ in precisely four entries \cite[Theorem 3.3.4]{KD}. In the case of latin squares defining groups the situation is more complicated. If two different Cayley table of order $n\ne\{4,6\}$ represent groups (not necessary distinct), then they differ from each other in at least $2n$ places. An arbitrary Cayley table of the cyclic group of order $4$ differs in at least four places from an arbitrary Cayley table of Klein's $4$-group (cf. \cite{Don, Drap, Fri, KD}).  

The interesting question is: {\em When two latin squares which differ in precisely four entries are isomorphic or isotopic}. To solve this problem we will use autotopies of the corresponding quasigroups, i.e., three bijections $\alpha$, $\beta$, $\gamma$ of a quasigroup $Q$ such that
$\alpha(x)\cdot\beta(y)=\gamma(x\cdot y)$ for all $x,y\in Q$.

\medskip
\noindent{\bf 2.} 
Let $Q=\{1,2,3,\ldots,n\}$ be a finite set. The multiplication (composition) of
permutations $\varphi$ and $\psi$ of $Q$ is defined as $\varphi\psi(x)=\varphi(\psi(x))$. All permutations will be
written in the form of cycles and cycles will be separated by points, e.g.
$$
\varphi=\left(\arraycolsep=1mm
\begin{array}{ccccccccc}
1 & 2 & 3 & 4 & 5 & 6&7&8\\
3 & 1 & 2 & 5 & 4 & 6&8&7
\end{array}
\right)=(132.45.6.78.)
$$

Permutations $L_i$ ($i\in Q$) of $Q$ such that $L_{i}(x)= i\cdot x$ for all $x\in Q$, are called {\em left translations} of an element $i$. Such permutations were firstly investigated by V. D. Belousov (cf. \cite{Bel}) in connection with some groups associated with quasigroups. Next, such translations were used by many authors to describe various properties of quasigroups, see for example \cite{D'91}, \cite{DD} or \cite{Sch}. Left translations are applied in \cite{Ind} to the construction of some polynomials that can be used to determine which quasigroups are isotopic. Namely, the main results of \cite{Ind} shows that isotopic quasigroups have the same polynomials.

\medskip
\noindent {\bf 3.}
We say that elements $x,y,z,u\in Q$, $x\ne z$, $y\ne u$, determine a {\em rectangle} in a quasigroup $Q$ if $xy=zu=a$ and $xu=zy=b$ for some $a,b\in Q$. Vertices of such rectangle have the form $xy$, $xu$, $zx$ and $zu$ (see example below). Such determined rectangle will be denoted by $\langle x,y,z,u\rangle$ or by $\langle x,u,z,y\rangle$. It is clear that $\langle x,u,z,y\rangle =\langle x,\,x\!\setminus\!a,\,a/(x\!\setminus\!b),\,x\!\setminus\!b\rangle$.

\begin{example}\label{Ex1}\rm
The following quasigroup
$$\begin{array}{c|cccccccc} 
\cdot&1&2&3&4&5&6&7\\ \hline 
    1&1&2&3&4&5&6&7\\ 
		2&2&1&\fbox{4}&3&7&5&\fbox{6}\\ 
		3&3&6&5&1&4&7&2\\
		4&4&5&2&7&6&1&3\\
		5&5&4&7&6&2&3&1\\
		6&6&7&1&2&3&4&5\\
		7&7&3&\fbox{6}&5&1&2&\fbox{4}
\end{array}$$
has the rectangle $\langle 2,3,7,7\rangle$. Other rectangles of this quasigroup are calculated in Example \ref{Ex3}.
\end{example}

\begin{theorem}
A group has a rectangle if and only if it has an element of order two.
\end{theorem}
\begin{proof}
Let $x$ be an element of a group $G$ such that $x^2=e$ and $x\ne e$. Then $ee=xx=e$, $ex=xe=x$, so $<e,e,x,x>$ is a rectangle (square).

Conversely, let $<x,y,z,u>$ be a rectangle in a group $G$. Then $xyu^{-1}=xuy^{-1}$, and consequently $(yu^{-1})^2=e$. So $a=yu^{-1}\ne e$ has order two.
\end{proof}

\begin{lemma}
Isotopic $($antiisotopic$)$ quasigroups have the same number of rectangles.
\end{lemma}
\begin{proof}
Let $(Q,\cdot)$ and $(Q,\circ)$ be isotopic quasigroups, i.e., $\gamma(x\cdot y)=\alpha(x)\circ\beta(y)$ for some bijections of $Q$ and all $x,y\in Q$. Then, as it is not difficult to see, $\langle x,y,z,u\rangle$ is a rectangle of $(Q,\cdot)$ if and only if $\langle\alpha(x),\beta(y),\alpha(z),\beta(u)\rangle$ is a rectangle of $(Q,\circ)$. 
\end{proof}

\medskip
\noindent{\bf 4.} Each quasigroup $Q=(Q,\cdot)$ determines five new quasigroups $Q_i=(Q,\circ_i)$ with the operations $\circ_i$ defined as follows:
$$
\begin{array}{cccc}
x\circ_1 y=z\longleftrightarrow x\cdot z=y\\
x\circ_2 y=z\longleftrightarrow z\cdot y=x\\
x\circ_3 y=z\longleftrightarrow z\cdot x=y\\
x\circ_4 y=z\longleftrightarrow y\cdot z=x\\
x\circ_5 y=z\longleftrightarrow y\cdot x=z\\
\end{array}
$$
Such defined (not necessarily distinct) quasigroups are called {\em parastrophes} or {\em conjugates} of $Q$. Traditionally they are denoted as

\medskip\centerline{ 
$Q_1=Q^{-1}=(Q,\backslash), \ \ Q_2={}^{-1}\!Q=(Q,/), \ \ Q_3= {}^{-1}\!(Q^{-1})=(Q_1)_2$,}

\medskip\centerline{$Q_4=({}^{-1}\!Q)^{-1}=(Q_2)_1$ \ and \ $Q_5=({}^{-1}\!(Q^{-1}))^{-1}=((Q_1)_2)_1=((Q_2)_1)_2$.}

From the above results and results obtained in \cite{Dud} it follows that {\em a fixed $IP$-quasigroup and all its parastrophes have the same number of rectangles. Similarly, a quasigroup isotopic to a group and its parastrophes.}

\medskip
\noindent{\bf 5.} 
Note that if a commutative quasigroup has a rectangle $\langle x,y,z,u\rangle$ with $x=y$, then also $z=u$. Similarly, $z=u$ implies $x=y$. So, if in a commutative quasigroup $Q$ one of vertices of the rectangle $\langle x,y,z,u\rangle$ lies on the diagonal of the multiplication table of $Q$, then this rectangle is a square and its diagonal coincides with the diagonal of the multiplication table.

In Boolean groups each two elements $x,y$ determine the rectangle $\langle x,y,xy,e\rangle$; each three elements $x,y,z$ determine the rectangle $\langle x,y,z,xyz\rangle$. 

An interesting question is how to find rectangles in a given quasigroup. Direct calculation of $\langle x,\,x\!\setminus\!a,\,a/(x\!\setminus\!b),\,x\!\setminus\!b\rangle$ is rather trouble. Below we present simplest method based on left translations.

Let $\langle x,y,z,u\rangle$ be a rectangle in a quasigroup $(Q,\cdot)$. Then, according to the definition, 
$xy=zu=a$ and $xu=zy=b$. Thus the left translations $L_x$ and $L_z$ have the form 
$$
L_x=\left(\arraycolsep=1mm
\begin{array}{ccccccccc}
\ldots & y &\ldots & u &\ldots\\
\ldots & a &\ldots & b &\ldots
\end{array}
\right), \ \ \ \ \ 
L_z=\left(\arraycolsep=1mm
\begin{array}{ccccccccc}
\ldots & y &\ldots & u &\ldots\\
\ldots & b &\ldots & a &\ldots
\end{array}
\right).
$$
Hence
$$
L_xL_z^{-1}=\left(\arraycolsep=1mm
\begin{array}{ccccccccc}
\ldots & b&\ldots & a &\ldots\\
\ldots & a &\ldots & b &\ldots
\end{array}
\right).
$$
This means that the permutation	$L_xL_z^{-1}$ used in the construction of indicators (for details see \cite{Ind}) has the cycle $(a,b)$. 
Thus vertices $a,b$ of this rectangle are located in the $x$-th row, vertices $b,a$ in the $a/(x\!\setminus\!b)$-th row. Since $L_zL_x^{-1}=(L_xL_z^{-1})^{-1}$, the permutation $L_xL_z^{-1}$ and $L_zL_x^{-1}$ have the same cycles of the length $2$. So, it is sufficient to calculate $L_xL_z^{-1}$ for $x<z$ only.

\begin{example}\label{Ex3}
The quasigroup presented in Example \ref{Ex1} has the following left translations: $L_1=(1.2.3.4.5.6.7.)$ and
$$
\begin{array}{llllll}
L_2=(12.34.576.),&&
L_3=(1354.267.),&&
L_4=(1473256.),\\[3pt]
L_5=(1524637.),&&
L_6=(1642753.),&&
L_7=(1745.236.).
\end{array}
$$
Consequently, 
$$
\begin{array}{llllll}
L_1L_2^{-1}=(12.34.576.),&&
L_2L_4^{-1}=(15.24.367.),&&
L_3L_4^{-1}=(17.25643.),\\[3pt]
L_4L_5^{-1}=(13.267.54.),&&
L_2L_7^{-1}=(17253.46.).&&
\end{array}
$$
In other $L_xL_z^{-1}$ with $x<z$,  there are  no cycles of the length $2$.

Below we present cycles $(a,b)$ and rectangles generated by these cycles.

$\begin{array}{lllllllll}
(1,2):\langle 1,1,2,2\rangle,&&&(3,4):\langle 1,3,2,4\rangle,&&&(1,5):\langle 2,2,4,6\rangle,\\[3pt]
(2,4):\langle 2,1,4,3\rangle,&&&(1,7):\langle 3,4,4,6\rangle,&&&(1,3):\langle 4,6,5,7\rangle,\\[3pt]
(5,4):\langle 4,1,5,2\rangle,&&&(4,6):\langle 2,3,7,7\rangle.
\end{array}
$

Hence this quasigroup has eight rectangles.
\end{example}

Note that one pair $(a,b)$ can determine several rectangles.

\begin{example}\label{Ex4}
In the quasigroup
$$\begin{array}{c|cccccccc} 
\cdot&1&2&3&4&5&6\\ \hline 
    1&1&2&3&4&5&6\\ 
		2&2&1&4&5&6&3\\ 
		3&3&4&2&6&1&5\\
		4&4&5&6&2&3&1\\
		5&5&6&1&3&2&4\\
		6&6&3&5&1&4&2
\end{array}$$
the pair $(1,2)$ determines three rectangles: $\langle 1,1,2,2\rangle$, $\langle 3,5,5,3\rangle$, $\langle 4,6,6,4\rangle$. Other rectangles are: $\langle 3,3,4,4\rangle$, $\langle 3,3,6,6\rangle$, $\langle 4,4,5,5\rangle$ and $\langle 5,5,6,6\rangle$.
\end{example}

\medskip
\noindent{\bf 6.} A quasigroup $(Q,\circ)$ is a {\em rectangle transformation} of a quasigroup $(Q,\cdot)$ if in $(Q,\cdot)$ there exists a rectangle $\langle a,b,c,d\rangle$ such that
$$
x\circ y=\left\{\begin{array}{ccl}
a\cdot d&&{\rm if } \ x=a,\ y=b,\\
a\cdot b&&{\rm if } \ x=c,\ y=b,\\
c\cdot b&&{\rm if } \ x=c,\ y=d,\\
c\cdot d&&{\rm if } \ x=a,\ y=d,\\
x\cdot y&&{\rm in \;other\;cases}.
\end{array}
\right.
$$

\begin{example}\label{Ex2}\rm
This quasigroup is obtained from the quasigroup given in Example \ref{Ex1} as a rectangle transformation by $\langle 2,3,7,7\rangle$:
$$\begin{array}{c|cccccccc} 
\circ&1&2&3&4&5&6&7\\ \hline 
    1&1&2&3&4&5&6&7\\ 
		2&2&1&\fbox{6}&3&7&5&\fbox{4}\\ 
		3&3&6&5&1&4&7&2\\
		4&4&5&2&7&6&1&3\\
		5&5&4&7&6&2&3&1\\
		6&6&7&1&2&3&4&5\\
		7&7&3&\fbox{4}&5&1&2&{\fbox{6}}
\end{array}$$
\end{example}

Two rectangles $\langle x,y,z,u\rangle$ and $\langle x',y',z',u'\rangle$ of a quasigroup $(Q,\cdot)$ are {\em equivalent} if there exists an autotopism $(\alpha,\beta,\gamma)$ of $(Q,\cdot)$ such that $\alpha(x)=x'$, $\beta(y)=y'$, $\alpha(z)=z'$ and $\beta(u)=u'$.

\begin{theorem}
A rectangle transformation by equivalent rectangles gives isotopic quasigroups.
\end{theorem}
\begin{proof} Let $(Q,\circ)$ and $(Q,*)$ be quasigroups obtained from $(Q,\cdot)$ as a rectangle transformation by $\langle a,b,c,d\rangle$ and $\langle \alpha(a),\beta(b),\alpha(c),\beta(d)\rangle$, where $(\alpha,\beta,\gamma)$ is an autotopy of a quasigroup $(Q,\cdot)$. Then $\gamma(a\circ b)=\gamma(a\cdot d)=\alpha(a)\cdot\beta(d)=
\alpha(a)*\beta(b)$. Analogously $\gamma(a\circ d)=\alpha(a)*\beta(d)$, $\gamma(c\circ b)=\alpha(c)*\beta(b)$ and $\gamma(c\circ d)=\alpha(c)*\beta(d)$. 
In other cases $x\circ y=x\cdot y=x*y$. So, quasigroups $(Q,\circ)$ and $(Q,*)$ are isotopic.
\end{proof}

\begin{example}\label{Ex5}
Applying the condition $\gamma L_i\beta^{-1}=L_{\alpha(i)}$, $i=1,2\ldots,7$, to the quasigroup defined in Example \ref{Ex1}, after long computations, we can see that this quasigroup has only one non-trival autotopism. It has the form $(\alpha,\beta,\gamma)$, where $\alpha =(15.24.37.6.)$, $\beta=(12.36.47.5.)$, $\gamma=(14.25.67.3.)$. Thus, the rectangle $\langle 1,1,2,2\rangle$ is equivalent to the rectangle $\langle \alpha(1),\beta(1),\alpha(2),\beta(2)\rangle=\langle 5,2,4,1\rangle=\langle 4,1,5,2\rangle$.
Also rectangles $\langle 1,3,2,4\rangle$ and $\langle 4,6,5,7\rangle$,  $\langle 2,1,4,3\rangle$ and $\langle 2,2,4,6\rangle$, $\langle 3,4,4,6\rangle$ and $\langle 2,3,7,7\rangle$ are pairwise equivalent. So, the quasigroup defined in Example \ref{Ex1} has four non-equivalent rectangles. Thus, by rectangle transformations, from this quasigroup we obtain four non-isotopic quasigroups.
\end{example}

Observe that by the converse rectrangle tansformations we obtain the same quasigroup. So, by the rectangle transformation from two non-isotopic quasigroups we can obtain isotopic quasigroups. Hence, the converse of Theorem 2 is not true.


\begin{thebibliography}{20}
\bibitem{Bel} {\bf V. D. Belousov}, {\em On group associated with a quasigroup} (Russian), Mat. Issled. {\bf 4} (1969), no. 3, $21- 39$.
\bibitem{D'91} {\bf I. I. Deriyenko}, {\em Necessary conditions of the isotopy of finite quasigroups}, (Russian) Mat. Issled. {\bf 120} (1991), $51-63$.
\bibitem{Ind} {\bf I. I. Deriyenko}, {\em Indicators of quasigroups}, Quasigroups Related Systems \textbf{19} (2011), $223-226$.
\bibitem{DD} {\bf I. I. Deriyenko, W. A. Dudek}, {\em D-loops}, Quasigroups Related Systems {\bf 20} (2012), $183-196.$
\bibitem{Don} {\bf D. Donovan, S. Oates-Williams and Ch. E. Praeger}, {\em On the distance between distinct group latin squares}, J. Combin. designs {\bf 5} (1997), 
$235-248.$
\bibitem{Drap} {\bf A. Dr\'apal}, {\em Hamming distances of groups and quasigroups}, Discrete Math. {\bf 235} (2001), $189-197.$
\bibitem{Dud} {\bf W. A. Dudek}, {\em Parastrophes of quasigroups}, Quasigroups Related Systems {\bf 23} (2015), $221-230.$

\bibitem{Fri} {\bf S. A. Frisch}, {\em On the minimal distance between group tables}, Acta Sci. Math. (Szeged), {\bf 63} (1997), $341-351$.
\bibitem{KD} {\bf D. Keedwell, J. D\'enes}, {\em Latin squares and their applications}, North-Holland, 2015.

\bibitem{Sch} {\bf V. Shcherbacov}, {\em Elements of quasigroup theory and applications}, CRC Press, 2017.
\end{thebibliography}
\end{document}